\documentclass[11pt]{article}
\usepackage{amsmath}
\usepackage{amsthm}
\usepackage{amssymb}
\usepackage{graphicx}
\usepackage{epsfig}
\usepackage[latin1]{inputenc}
\usepackage{color}

\def\subneq{\mathop{\raise 0.7ex \hbox{$\subset$}}\!\!\!\!\!\!{\raise -0.6ex\hbox{$\neq$}}\,}

\def\u{\underline}

\def\QQ{{\ \rlap {\raise 0.4ex \hbox{$\scriptscriptstyle |$}}\hskip -0.2em Q}}

\def\1{{1\hskip-0.25em{\rm l}}}

\def\CC{{\ \rlap{\raise 0.4ex \hbox{$\scriptscriptstyle |$}}\hskip -0.2em C}}

\def\sobre#1#2{\lower 1ex \hbox{ $#1 \atop #2 $ } }
\def\bajo#1#2{\raise 1ex \hbox{ $#1 \atop #2 $ } }

\def\ep{\varepsilon}

\def\p{\partial}

\def\u2{{u^\ep \over \ep^2 }}
\def\u3{{\displaystyle {\bar u}^\ep \over \ep^2 }}

\begin{document}

\markboth{Eduard Feireisl, Philippe Lauren\c{c}ot and Andro Mikeli\'c}{Global-in-time solutions for the isothermal Matovich-Pearson equations}

\newtheorem{theorem}{Theorem}
\newtheorem{lemma}{Lemma}
\newtheorem{proposition}{Proposition}
\newtheorem{corollary}{Corollary}
\newtheorem{remark}{Remark}
\newtheorem{definition}{Definition}
\title{Global-in-time solutions for the isothermal Matovich-Pearson equations \thanks{This paper was partly written during the visit of Ph.L. and A.M.  at {\bf Ne\v cas Center for Mathematical Modelling} in March/April 2008  and they wish to express their thanks for the hospitality during  the stay.  }}

\author{\bf Eduard Feireisl \thanks{The work of E.F. was supported by Grant 201/09/0917 of GA \v CR as a part of the general research programme of the Academy of Sciences of the Czech Republic, Institutional Research Plan AV0Z10190503.}\\ Institute of Mathematics of the Academy of Sciences\\ of the Czech Republic, \v Zitn\' a 25, 115 67 Praha 1, \\ Prague, CZECH REPUBLIC\\ E-mail: {\tt feireisl@math.cas.cz}  \and \bf Philippe Lauren\c{c}ot \\ Institut de Math\'ematiques de Toulouse, CNRS UMR~5219\\ Universit\'e de Toulouse, F--31062 Toulouse Cedex 9, FRANCE\\ E-mail: {\tt laurenco@math.univ-toulouse.fr} \and \bf  Andro Mikeli\'c \thanks{Research of A.M. on the mathematical modeling of glass fiber drawing was supported in part by the EZUS LYON 1 INGENIERIE during the period 2004 -- 2007.} \\ Universit\'e de Lyon, Lyon, F-69003, FRANCE ;\\ Universit\'e Lyon 1, Institut Camille Jordan, CNRS UMR 5208,\\  D\'epartement de  Math\'{e}matiques, 43, Bd du 11 novembre 1918,\\ 69622 Villeurbanne Cedex,  FRANCE \\ E-mail: {\tt Andro.Mikelic@univ-lyon1.fr}}

\maketitle

\begin{abstract}
In this paper we study the Matovich-Pearson equations describing the process of glass fiber drawing. These equations may be viewed as a 1D-reduction of the incompressible Navier-Stokes equations including  free boundary, valid for the drawing of a long and thin glass fiber. We concentrate on the isothermal case without surface tension. Then the Matovich-Pearson equations represent a nonlinearly coupled system of an elliptic equation for the axial velocity and a hyperbolic transport equation for the fluid cross-sectional area. We first prove existence of a local solution, and, after constructing appropriate barrier functions, we deduce that the fluid radius is always strictly positive and that the local solution remains in the same regularity class. To the best of our knowledge, this is the first global existence and uniqueness result for this important system of equations.

\bigskip

{\bf Keywords:} Fiber drawing; Matovich-Pearson equations; incompressible free boundary Navier-Stokes equations;  non-local transport equation; iterated comparison; barrier function.

\bigskip
{\bf Mathematics Subject Classification MSC2000 (AMS classcode):} 35Q30;  35Q35; 35R35; 35L80; 76D05; 76D27


\end{abstract}

\section{Introduction}

The drawing of continuous glass fibers is a widely used procedure. Industrial glass fibers are manufactured by a bushing with more than thousand nozzles. Bushings are supplied with a molten glass from a melting furnace. Its temperature ranges from $1300K$ to $1800K$. In order to understand the glass fiber forming process, it is important to study the drawing of a single glass fiber. This is, of course, a significant simplification because we disregard interaction between fibers and between fibers and the surrounding air. For a single glass fiber, the hot glass melt is forced by gravity to flow through a die into air. After leaving the die, the molten glass forms a free liquid jet. It is cooled and attenuated as it proceeds through the air. Finally, the cold fiber is collected on a rotating drum.

The molten glass can be considered as a Newtonian fluid and the process is described by the non-isothermal Navier-Stokes equations for a thermally dilatable but isochoric fluid. Since we deal with a free liquid jet, the problem is posed as a free boundary problem for the Navier-Stokes equations, coupled with the energy equation. We refer to \cite{FFM:07} for detailed modeling and analysis of the equations describing the stationary flow inside the die. There are several models proposed to describe the various
stages of the flow of a molten glass from the furnace to the winding spool: the slow flow in the die (the ``first phase" of the drawing), the jet formation under rapid cooling (the ``second phase"), and the terminal fiber profile (the ``third phase") (see  \cite{FMF:09}).

Since we consider long (their length is approximately $10$ m) and thin (their diameter varies from $1$ mm to $10$ $\mu$m) fibers, it is reasonable to apply the lubrication approximation to the model equation. This approach yields good results, at least for flows far from the die exit and in the so-called ``third phase" of the fiber drawing.

A standard engineering model for the isothermal glass fiber drawing in the ``third phase" is represented by the \emph{Matovich-Pearson equations}.  For an axially symmetric fiber with a straight central line, they read
\begin{equation}\label{Matov}
\partial_t  \textbf{A} + \partial_x (v \textbf{A}) =0 ; \qquad \partial_x  (3 \mu (T) \textbf{A}\ \partial_x v) + \partial_x  ( \sigma (T) \sqrt{\textbf{A}} ) =0,
\end{equation}
where $\textbf{A} = \textbf{A} (t,x)$ is the cross-area of the fiber section, $v = v(t,x)$ is the effective axial velocity, $3 \mu  $ is Trouton's viscosity, and $\sigma$ denotes the surface tension. As the coefficients $\mu$ and $\sigma$ depend on the temperature, it is necessary to take into account an equation for the temperature $T = T (t,x)$.\vskip2pt

The original derivation of the system (\ref{Matov}) is purely heuristic and obtained under the assumptions that: \textit{(i)}  the viscous forces dominate the inertial ones; \textit{(ii)} the effect of the surface tension is balanced with the normal stress at the free boundary; \textit{(iii)} the heat conduction is small compared with the heat convection in the fiber; \textit{(iv)} the fiber is almost straight, and all quantities are axially symmetric. We refer to the classical papers by Kase \& Matsuo \cite{KaseMat:65, KaseMat:67}, and Matovich \& Pearson \cite{MatPear:69} for more details concerning the model.

Another derivation of the model based on a lubrication type asymptotic expansion can be found in the work by Schultz \textit{et al.} \cite{GuSch:98,  SchuDa:82}, Dewynne \textit{et al.} \cite{DeOckWil:92,DeHowWil:94},  and Hagen \cite{HagenZAMM02}, with more emphasis on the mathematical aspects of the problem. The (formal) asymptotic expansion is developed with respect to a small parameter $\varepsilon$, proportional to the ratio of the characteristic thickness $R_E$ in the radial direction and the characteristic axial length of the fiber $L$. The fact that the viscosity changes over several orders of magnitude is surprisingly ignored in these studies. As a matter of fact, the viscosity coefficient depends effectively on the temperature, with values varying from $10$ to $10^{12}$ Pa sec, while in the above mentioned asymptotic expansions it is considered simply of order one. A correct formal derivation was given in \cite{CloFaFaMik:08}, and it is in full agreement with the model announced in \cite{HagenZAMM02}.  Finally, a full non-stationary model of a thermally dilatable molten glass, with density depending on the temperature, was derived in \cite{FMF:09}.

A mathematical analysis of generalized \emph{stationary} Matovich-Pearson equations is performed in  \cite{CloFaFaMik:08} (see also \cite{DewOckWil:89}). The non-stationary case, without surface tension and with advection equation for the temperature, is studied by Hagen \& Renardy  \cite{HagenRenardy99}. They prove a local-in-time existence result in the class of smooth solutions.
Their approach is based on a precise analysis of the dependence of the solution of the mass conservation equation on the velocity. This method  requires controlling higher order Sobolev norms in the construction of solutions by means of an iterative procedure and works only for short time intervals. Hagen \textit{et al.} \cite{HagenRenardy01, Hagen05, HagenJEE05} have also undertaken a detailed study of the linearized equations of  forced elongation. Despite this considerable effort, global-in-time solvability of the Matovich-Pearson equations was left open.

The ratio $\sigma /\mu$ is small, and, furthermore, the inertia and gravity effects are negligible in most applications. Accordingly, we consider the Matovich-Pearson equations (\ref{Matov}) with $\sigma =0$, meaning, the isothermal drawing with constant positive viscosity and in the absence of surface tension. For a prescribed velocity at the fiber end points, the important parameter is the {\it draw ratio}, being the ratio between the outlet and inlet fluid velocities. It is well known that the instability known as a draw resonance occurs at draw ratios in excess of about $20.2$. Linear stability analysis was rigorously undertaken by Renardy \cite{Renardy:06}. Moreover, in \cite{Yarin99}, it was established that the cross section, given by the Matovich-Pearson equations with $\sigma=0$, may vary chaotically at a draw ratio higher than $30$, under the condition of periodic variations of the input cross section. There are also numerous articles devoted to numerical simulations confirming such a conclusion. Fairly complete simulations can be found in the papers by Gregory Forest \& Zhou \cite{GforestZhou:01, ZhouForest05}. Their simulations predict various aspects of the physical process, like a linearized stability principle, bounds on the domain of convergence for linearly stable solutions, and transition to instability. Their analysis completes that of \cite{GeyHomsy:80}.

The above mentioned simplification of system (\ref{Matov}) is briefly discussed in \cite{HagenNA05}, however, without any rigorous proofs. Our idea is to use the particular structure of the system with $\sigma =0$, and to prove short-time existence of smooth solutions satisfying good  uniform estimates. Then, performing a qualitative analysis  of the solutions and constructing  appropriate barrier functions, we show that the cross-section area remains bounded below away from zero. This observation allows us to deduce existence as well as uniqueness of global-in-time solutions.

\section{Isothermal fiber drawing without surface tension}

We study the system of equations
\begin{gather}
\partial_t A + \partial_x(v A) =0 \quad \mbox{ in } \quad Q_T = (0,T)\times (0,L), \label{Mat1} \\
\partial_x (3\mu A\ \partial_x v ) =0 \quad \mbox{ in } \quad Q_T = (0,T)\times (0,L), \label{Mat2}
\end{gather}
supplemented with the boundary and initial conditions
\begin{gather}
    A(t,0) = S_0 (t) \; \mbox{ in } \; (0,T), \quad A(0,x)= S_1 (x) \quad \mbox{ in } \; (0,L), \label{Mat3} \\
    v(t,0) = v_{in} (t) \; \mbox{ in } \; (0,T), \quad v(t, L)= v_L (t) \quad \mbox{ in } \; (0,T). \label{Mat4}
\end{gather}
Here $v$ is the axial velocity and $A$ denotes the cross section, $L$,  $T$ are given positive numbers, and $3 \mu > 0$ denotes Trouton's viscosity assumed to be constant.

The data satisfy
\begin{equation} \label{Data1-}
\left\{ \begin{array}{l}
0 < v_m \leq v_{in} (t) < v_L(t) \leq V_M \ \mbox{for any} \ t \in (0,T),\\ \\
0 < S_m \leq S_0(t), S_1(x) \leq S_M \ \mbox{for all} \ (t,x) \in Q_T,\ S_0(0) = S_1(0).
\end{array}
\right.
\end{equation}
Moreover, the functions $v_{in}$, $v_L$, $S_0$, $S_1$ belong to certain regularity classes specified below.

\subsection{A priori bounds}

Our construction of global-in-time solutions is based on certain {\it a priori} estimates that hold, formally, for any smooth solution
of problem (\ref{Mat1}) - (\ref{Mat4}), with the cross-section area $A > 0$.
The crucial observation is that, as a direct consequence of (\ref{Mat2}),
\begin{equation} \label{pom1}
A(t,x) \ \partial_x v (t,x) = \chi (t)\ \mbox{for any} \ t\in (0,T),
\end{equation}
where $\chi$ is a function of the time variable only. Moreover, as $A$ is positive and the axial velocity satisfies the boundary conditions
(\ref{Data1-}), we deduce that
\begin{equation} \label{pom2}
\chi (t) > 0 \ \mbox{for any} \ t\in (0,T),
\end{equation}
which in turn implies
\begin{equation}\label{pom3}
\partial_x v (t,x) > 0 \ \mbox{for all}\ (t,x)\in Q_T.
\end{equation}
Accordingly,
\[
v_{in}(t) < v(t,x) < v_L (t) \ \mbox{for all} \ (t,x)\in Q_T.
\]

Next, we rewrite equation (\ref{Mat1}) in the form
\[
\partial_t A + v \partial_x A = - \chi \le 0
\]
yielding
\begin{equation} \label{pom4}
A(t,x) \leq S_M \ \mbox{for all} \ (t,x)\in Q_T.
\end{equation}

 Integrating (\ref{pom1}) over $(0,L)$ and using (\ref{Mat4}) and (\ref{pom4}) give rise to the uniform bound
\begin{equation} \label{pom5}
0 < \chi(t) < \frac{ v_L (t) - v_{in} (t) }{L} S_M \ \mbox{for all} \ t\in (0,T).
\end{equation}

In order to deduce a lower bound for the cross section area $A$, we first observer that $A$ and $\partial_x A$ satisfy
the same transport equation, namely,
\begin{gather}
\partial_t A + \partial_x (v A) =0, \label{pom6} \\
\partial_t \left( \partial_x A \right)  + \partial_x \left( v \left( \partial_x A \right) \right) =0. \label{pom7}
\end{gather}
In particular,
\begin{equation} \label{pom9}
\partial_t \left( \partial_x \log(A) \right)  + v \partial_x \left( \partial_x \log(A) \right) =0.
\end{equation}

Evaluating the boundary values of $\partial_x \log(A)$  with (\ref{Mat3}), namely,
\[
\partial_x \log(A) (t,0) = \frac{\partial_x A(t,0)}{S_0 (t)}, \ \partial_x \log(A) (0,x) = \frac{\partial_x S_1(x)}{S_1 (x)},
\]
where, in accordance with (\ref{Mat1}), (\ref{Mat3}), (\ref{Mat4}), (\ref{pom1}), and (\ref{pom5})
\begin{eqnarray*}
\partial_x A(t,0) & = & - \frac{1}{v_{in}(t)}\ \left( \chi(t) + \frac{dS_0}{dt} (t) \right) \\
& \ge & - \frac{1}{v_{in}(t)}\ \left( \frac{v_L(t)-v_{in}(t)}{L} S_M + \frac{dS_0}{dt} (t)\right) .
\end{eqnarray*}
We deduce easily the desired lower bound on $A$
\begin{equation} \label{pom10}
A(t,x) \geq {A}_m > 0 \ \mbox{for all} \ (t,x)\in Q_T,
\end{equation}
where the constant $A_m$ is determined solely in terms of $v_m$, $V_M$, $S_m$, $S_M$, and the first derivatives of $S_0$, $S_1$.

The {\it a priori} bounds derived in (\ref{pom1}) - (\ref{pom10}) form a suitable platform for the existence theory developed in the remaining part of this paper.

\section{Short time existence of regularized strong solutions}

In addition to (\ref{Data1-}), we shall assume that
\begin{equation}\label{Data1}
S_0 \in W^{2,\infty} (0,T) , \quad  S_1 \in H^2 (0,L), \quad v_{in} , v_L \in W^{1, \infty} (0,T),
\end{equation}
where the symbol $W^{k,p}$ denotes the standard Sobolev space of functions having $k-$derivatives $L^p-$integrable, and
$H^2 \equiv W^{2,2}$.  For further use, we introduce the notation
\begin{equation}\label{Compcte}
Q^{0, 0} = -\frac{d S_0}{d t} (0) - v_{in} (0) \frac{d S_1}{d x} (0).
\end{equation}

Let us begin with a list of definitions:

\begin{definition} Let $t_0$ be a positive number. A pair $(A, v)$, defined on $Q_{t_0} = (0, t_0 ) \times (0,L)$,  is a strong solution of (\ref{Mat1})-(\ref{Mat4}) if
\begin{gather}
A\in W^{1, \infty} ((0,t_0) \times (0,L)), \label{Msol1a} \\
v ,\ \partial_t v,\ \partial_x v,\ \partial_{tx}^2 v,\ \partial_x^2 v \in L^\infty((0,t_0) \times (0,L)), \label{Msol1b} \\
(A, v) \; \; \mbox{ satisfy equations } (\ref{Mat1})-(\ref{Mat2})\ \mbox{a. e. in} \ (0,t_0) \times (0,L), \label{Msol2} \\
A > 0 \; \; \mbox{ on } \; Q_{t_0} \; \mbox{ and } \; A \; \mbox{ satisfies }  (\ref{Mat3}) \; \mbox{ pointwise }, \label{Msol3} \\
v \; \;  \mbox{ satisfies } (\ref{Mat4}) \; pointwise. \label{Msol4}
\end{gather}
\end{definition}

\begin{definition} \label{set1} Let $\ t_0 $ be a constant, $0< t_0 \leq T$.
For $h \in  L^{\infty} (0, t_0 ; H^1 (0,L))$  with $\partial_x h \in L^\infty(0,L;L^2 (0, t_0))$, we define the energy functional $\mathcal{E}$ as
\begin{gather}
\mathcal{E}(h)^2  = {\rm ess}\sup_{0 < t < t_0 } || h(t, \cdot) ||_{H^1 (0,L)}^2 + v_m \ {\rm ess} \sup_{0 < x < L }  ||  \partial_x h(\cdot , x) ||_{L^2 (0, t_0 )}^2 \label{Energ1}.
\end{gather}

Similarly, the radius $R$ is defined by
\begin{equation}\label{Radius}
 \frac{1}{8} R^2 =  \int^L_0 \left\{ S_1^2 + \left( \frac{d S_1 }{dx} \right)^2 \right\} ( x ) \ dx +  \int^T_0  v_{in}(t) \left\{ S_0^2 + \frac{2}{ v_{in}^2} \left| \frac{ d S_0 }{dt} \right|^2 \right\} (t) \ dt .
\end{equation}
\end{definition}

\begin{definition} \label{set2}
For $R$  given by (\ref{Radius}) and $0< t_0 \leq T$, we denote by $\mathcal{S} (t_0 , R)$ the convex set of nonnegative functions $h$ defined on $Q_{t_0}$ such that
\begin{gather}
    h \in W^{1, \infty} (0, t_0 ; L^2 (0,L))\cap L^{\infty} (0, t_0 ; H^1 (0,L)) \cap  \notag \\
    \cap  W^{1, \infty} (0, L; L^2 ( 0, t_0))  \cap L^{\infty} (0, L;H^1 ( 0, t_0)),
  \label{Conv1} \\
\mathcal{E} (h) \leq R, \quad \mbox{and} \notag \\
{\rm ess}  \sup_{ 0 < t < t_0 }   || \p_t h (t,\cdot) ||_{L^2 (0,L)}
  + \sup_{x\in [0, L] } || \p_t h (\cdot , x) ||_{L^2 (0, t_0)}  \leq \notag \\
 R  V_M \left( 2+\frac{1}{L} + \frac{\sqrt{T}}{L^{3/2}} \right) +  | Q^{0,0} | \left( \sqrt{T} + \sqrt{L} \right)  \label{Conv2} \\
h(0, x) = S_1 (x) \quad \mbox{ and } \quad h(t, 0) = S_0 (t). \label{Conv3}
\end{gather}
\end{definition}

\begin{definition} \label{set3}
For $R$  given by (\ref{Radius}),
$$
 \alpha = \frac{4}{v_m} |Q^{0,0}|^2 + 4S_M L |Q^{0,0}| ,\quad \beta = \frac{4 V_M^2 R^2}{v_m L^3} + 4 S_M^2 V_M + \frac{4 S_M V_M R}{\sqrt{L}} ,
$$
we denote $t^* \in (0,T]$ a positive time satisfying
\begin{gather}
t^* \le \frac{S_m L^{3/2}}{4RV_M} \quad \mbox{and} \quad  t^*\leq \frac{R^2}{8(\alpha+\beta)}    . \label{Small2}
\end{gather}
\end{definition}

In this section, for $\delta>0$ small enough, we construct a family of approximate solutions $(A^\delta,v^\delta)$ solving the following initial-boundary value problem:
\begin{gather}
\partial_t A^\delta + \partial_x (v^\delta A^\delta ) =0 \quad \mbox{ in } \quad Q_{t_0}^{\delta} = (\delta ,t_0 ) \times (0,L) , \label{MatShortReg1} \\
\partial_t A^\delta + v^\delta \partial_x A^\delta + Q^{0,0} \left( 1-\frac{t}{\delta} \right) + \notag \\
\frac{v_L (\delta ) - v_{in} (\delta )}{\displaystyle \int^L_0 \frac{d\xi }{A^\delta (\delta , \xi )}} \frac{t}{\delta} =0 \quad \mbox{ in } \quad  (0, \delta ] \times (0,L) , \label{MatShortReg1a} \\
    \partial_x (3\mu A^\delta \ \partial_x v^\delta ) =0 \quad \mbox{ in } \quad Q_{t_0} = (0, t_0)\times (0,L), \label{MatShortReg2} \\
    A^\delta (t,0) = S_0 (t) \; \mbox{ in } \; (0,T), \quad A^\delta (0,x)= S_1 (x) \quad \mbox{ in } \; (0,L), \label{MatShortReg3} \\
    v^\delta (t,0) = v_{in} (t) \; \mbox{ in } \; (0,T), \quad v^\delta (t, L)= v_L (t) \quad \mbox{ in } \; (0,T). \label{MatShortReg4}
\end{gather}

Specifically, we prove the following result.

\begin{theorem}\label{SmallTimeExist} Let $v_L, v_{in} , S_0 $ and $S_1$ satisfy (\ref{Data1}).  Consider $t^* >0$ given by (\ref{Small2}), $t_0\in (0,t^*)$, and let $\delta \in (0, \min{\{1,t_0\}} )$ be a small number satisfying
\begin{equation}\label{DeltaCo}
    \frac{S_m}{12}  - \delta | Q^{0,0} | >0.
\end{equation}

Then the initial-boundary value problem (\ref{MatShortReg1}) - (\ref{MatShortReg4}) possesses a unique solution $(A^\delta,v^\delta)$ in the class
\begin{gather}
    A^\delta \in C^1 ([0, t_0 ] ; H^1 (0, L)) \cap C([0, t_0 ]; H^2 (0,L) ) ,\label{SolMatReg1}\\
  v^\delta \in C^1 ([0, t_0 ] ; H^2 (0, L)) \cap C([0, t_0 ]; H^3 (0,L) ).\label{SolMatReg2}
\end{gather}
\end{theorem}

Note that for a strictly positive $h\in \mathcal{S} (t_0 , R)$, the corresponding velocity field  $v$ solving $\partial_x (h \partial_x v)=0$ with boundary conditions (\ref{Mat4}) reads
\begin{gather}
 v(t,x) = v_{in} (t) + \frac{v_L (t) - v_{in} (t)}{\displaystyle \int^L_0  \frac{d\xi }{h (t, \xi )}} \int^x_0  \frac{d\xi }{h (t, \xi )} , \label{Veloc1} \\
 \p_x v(t,x) =  \frac{v_L (t) - v_{in} (t)}{\displaystyle \int^L_0  \frac{d\xi }{h (t, \xi )}}  \frac{1 }{h (t, x) } .
\label{Veloc2}
\end{gather}

\begin{remark}  Under the compatibility condition
\begin{equation}\label{Compat1}
Q^{0,0} = -\frac{d S_0}{d t} (0) - v_{in} (0) \frac{d S_1}{d x} (0) = \frac{v_L (0) - v_{in} (0)}{\displaystyle \int^L_0 \frac{dx}{S_1 ( x )}} ,
\end{equation}
we could set $\delta =0$. However, imposing (\ref{Compat1}) is not physically justified. Note that this condition is systematically used in \cite{HagenRenardy99} as well as \cite{HagenNA05}.
\end{remark}

The rest of this section is devoted to the proof of Theorem \ref{SmallTimeExist}.
The basic and rather standard idea is to construct a sequence approaching a fixed point of a suitable nonlinear mapping. The proof is carried over by means of several steps.  We fix $t_0\in (0,t^*)$ and  $\delta \in (0, \min{\{1,t_0\}} )$ such that (\ref{DeltaCo}) is satisfied.

\vskip10pt
\fbox{STEP~1}
\vskip10pt
We take an arbitrary $A^0\in \mathcal{S} (t_0 , R)$. Then we substitute $h=A^0$ into (\ref{Veloc1}) and calculate $v^0 =v$. We note that $A^0\in \mathcal{S} (t_0 , R)$ implies
\begin{gather}
v^0 \in W^{1, \infty} ([0, t_0] ; H^1 (0,L))\cap L^{\infty} ([0, t_0] ; H^2 (0,L)),
\label{Itervel1} \\
t\mapsto \p_x v^0 ( t , 0) = \frac{v_L (t) - v_{in} (t)}{\displaystyle \int^L_0  \frac{d\xi }{A^0 (t, \xi )}}  \frac{1 }{S_0 (t) } \in H^1 (0, t_0 ) , \label{Itervel2} \\
  \p_x v^0 (t,x) >0  \quad \mbox{and} \quad v_{in} (t) \leq v^0 (t, x) \leq v_L (t) \; \mbox{ in } \;    [0, t_0 ]\times [0, L].\label{Itervel3}
\end{gather}
Next, we introduce functions $Q^{0}$ and $Q^{0, \delta}$ defined on $[0, t_0]$ by
\begin{gather}
Q^{0} (t) = \frac{v_L (t) - v_{in} (t)}{\displaystyle \int^L_0 \frac{d\xi }{A^0 (t, \xi )}}   \label{Force1} \\
Q^{0, \delta} (t) = \left\{
                      \begin{array}{ll}
                        \displaystyle Q^{0} (t), \quad \hbox{for} \; \delta \leq t\leq t_0 , &\\ \\
                        \displaystyle Q^{0, 0} + ( Q^{0} (\delta ) - Q^{0, 0} ) \frac{t}{\delta}, \quad \hbox{for} \; 0\leq t< \delta .&
                      \end{array}
                    \right.
\label{Force2}
\end{gather}
Obviously, $Q^{0, \delta}  \in W^{1, \infty } (0, t_0 )$, and, using Jensen's inequality  and (\ref{Data1-}), we get
\begin{eqnarray}
0 \le Q^{0} (t) & \leq & \frac{ v_L (t) }{\displaystyle \int^L_0 \frac{dx }{A^0 (t, x )}} \leq \frac{v_L (t)}{L^2}
\int^L_0 A^0 (t, x ) \ dx \nonumber \\
& \leq & \frac{V_M}{L^{3/2}} || A^0 (t, \cdot )||_{L^2 (0,L)} \le \frac{V_M R}{L^{3/2}}, \label{EstDervel}
\end{eqnarray}
since $A^0\in \mathcal{S}(t_0,R)$, as well as
\begin{equation}\label{EstDervela}
\left| Q^{0,\delta}(t) \right| \le |Q^{0,0}|\ \left( 1 - \frac{t}{\delta} \right)_+ + \frac{V_M R}{L^{3/2}}
\end{equation}

Now we are in a position to define the solution operator:

For $A^0 \in  \mathcal{S} (t_0 , R)$ and $v^0$ given by (\ref{Veloc1})  with $A^0$ instead of $h$, we solve the initial-boundary value problem
\begin{gather}
\partial_t u(t,x) = - v^0 (t,x) \partial_x u(t,x) - Q^{0, \delta} (t) , \quad (t,x) \in  Q_{t_0}, \label{Solop1} \\
u(0,x) = S_1 (x), \; x\in (0,L), \qquad u(t,0) = S_0 (t), \; t\in (0,t_0 ). \label{Solop2}
\end{gather}
Because of (\ref{Itervel1})-(\ref{Itervel3}) and regularity and compatibility of the data, we may apply a result of Hagen \& Renardy \cite{HagenRenardy99} (see Theorem \ref{HagenR} in Appendix) to problem (\ref{Solop1})-(\ref{Solop2}) to obtain a unique solution
    $u\in C^1 ([0, t_0 ] ; H^1 (0, L)) \cap C([0, t_0 ]; H^2 (0,L) ).$
We set
\begin{equation}\label{Solop}
    A^1 (t,x) = u (t,x) , \quad (t,x)\in Q_{t_0} .
\end{equation}
Relation (\ref{Solop}) defines a nonlinear operator, assigning to a given $A^0 \in \mathcal{S} (t_0 , R)$ the unique function $A^1$ in the class $C^1 ([0, t_0 ] ; H^1 (0, L)) \cap C([0, t_0 ]; H^2 (0,L) ).$ By (\ref{Solop1}), $\p^2_{t} u  + \p_t Q^{0, \delta} \in  L^\infty(Q_{t_0})$. Using a compactness lemma of Aubin type (see e.g. \cite{SimonJ:87}) we conclude that our nonlinear operator, defined on  $\mathcal{S} (t_0 , R)$, is compact.
\vskip5pt
At this stage we  follow the ideas of \cite{HagenRenardy99} and \cite{HagenNA05} to show that this nonlinear operator has a fixed point for all $t_0 < t^*$. They study the nonisothermal fiber spinning and their system of equations is different. Hence we are obliged to give an independent proof of short time existence, but the result remains close to their considerations.

The natural approach is to apply Schauder's fixed point theorem. To use it we have to prove that $\mathcal{S} (t_0 , R)$ is  a relatively compact convex set invariant under our nonlinear mapping. Finally, we should establish the  continuity of the mapping (\ref{Solop}).

\vskip10pt
\fbox{STEP~2}
\vskip10pt

We establish uniform $L^\infty-$ bounds for the function $u$. We have

\begin{lemma} Let $R$ be given by Definition \ref{set2}  and let $t_0 \leq t^*$.
Then the solution $u$ of problem (\ref{Solop1})-(\ref{Solop2}) satisfies the estimate
\begin{gather}
 \frac{S_m}{4} \leq u(t,x) \leq  S_M + \frac{2 S_m}{3} \le 2S_M  \, \mbox{ in } \, [0,t_0 ]\times [0,L].\label{Posit1}
\end{gather}
\end{lemma}

\begin{proof} Let $s_m\in W^{2, \infty} (0, t_0 )$ be the solution to the Cauchy problem
\begin{equation}\label{LowCauchy}
    \frac{d s_m (t)}{d t} = - Q^{0, \delta} (t) \; \mbox{in } \; (0, t_0 ), \, s_m (0) =  \frac{2 S_m}{3} .
\end{equation}
 Then $s_m$ clearly solves (\ref{Solop1}) and it follows from (\ref{Data1-}) and (\ref{Solop2}) that
$$
u(0,x) = S_1(x) \ge S_m \ge s_m(0) , \quad x\in (0,L) ,
$$
while (\ref{Data1-}), (\ref{DeltaCo}), the nonnegativity of $Q^0$, (\ref{Solop2}), and (\ref{LowCauchy}) ensure that
\begin{eqnarray*}
s_m(t) & = & \frac{2 S_m}{3} - \int_0^t Q^{0,\delta}(s) ds \le \frac{2 S_m}{3} - \int_0^{\min{\{\delta , t\}}} Q^{0,\delta}(s) ds \\
& \le & \frac{2 S_m}{3} - \int_0^{\min{\{\delta , t\}}} Q^{0,0} \left( 1 - \frac{s}{\delta} \right) ds \le \frac{2 S_m}{3} + |Q^{0,0}| \frac{\delta}{2} \\
& \le & S_m \le S_0(t) = u(t,0) .
\end{eqnarray*}
The comparison principle then entails that $u(t,x)\ge s_m(t)$ for $(t,x)\in Q_{t_0}$, whence, by (\ref{Small2}), (\ref{DeltaCo}), and (\ref{EstDervela}),
\begin{eqnarray*}
u(t,x) & \ge & s_m(t) \ge \frac{2 S_m}{3} - \int_0^t |Q^{0,\delta}(s)| ds  \\
& \ge & \frac{2 S_m}{3} - |Q^{0,0}| \frac{\delta}{2} - \frac{V_M R}{L^{3/2}}\ t^* \ge \frac{S_m}{4}.
\end{eqnarray*}

 This proves the lower bound. Proving the upper bound is analogous. The comparison function is now given by the solution to the Cauchy problem
\begin{equation}\label{LowCauchy1}
\frac{d s_M (t)}{d t} = - Q^{0, \delta} (t) \; \mbox{in } \; (0, t_0 ), \, s_M (0) = S_M + \frac{S_m}{2} .
\end{equation}
 Indeed, $s_M$ clearly solves (\ref{Solop1}) with $s_M(0)\ge S_M\ge S_1(x)=u(0,x)$ for $x\in (0,L)$ and it follows from (\ref{Data1-}), (\ref{Small2}), (\ref{DeltaCo}), (\ref{EstDervela}), (\ref{Solop2}), and (\ref{LowCauchy1}) that
$$
s_M(t) \ge s_M(0) - |Q^{0,0}| \frac{\delta}{2} - \frac{V_M R}{L^{3/2}}\ t^* \ge S_M \ge u(t,0) .
$$
Applying again the comparison principle gives $u(t,x)\le s_M(t)$ for $(t,x)\in Q_{t_0}$, which completes the proof since (\ref{DeltaCo}) guarantees that
$$
s_M(t) \le S_M + \frac{S_m}{2} + |Q^{0,0}| \frac{\delta}{2} \le S_M + \frac{2 S_m}{3} .
$$
This proves the Lemma.
\end{proof}

\medskip

Now we use equation (\ref{Solop1}) to calculate $G(t) = \p_x u (t, 0) \in C([0, t_0 ])$ getting:
\begin{equation}\label{InitDer}
 G(t) = \p_x u (t, 0)  = -\frac{1}{v_{in} (t)} \left( \frac{d S_0 (t)}{d t} + Q^{0,\delta } (t) \right).
\end{equation}
Due to the assumptions on the data, $G \in W^{1, \infty } (0, t_0 )$ and we have
\begin{equation}\label{EstG}
    | G(t) | \leq \frac{1}{ v_{in} (t) } \left(  | \frac{d S_0 (t)}{d t} | + | Q^{0,\delta } (t) | \right).
\end{equation}

Next, we take the derivative of  equation (\ref{Solop1}) with respect to the $x$ variable. This yields  that $S=\partial_x u$ solves
\begin{gather}
\partial_t S = - v^0 \partial_x S - \p_x v^0 S , \quad (t,x) \in  Q_{t_0}, \label{SolopD1} \\
S (0,x) = \frac{ d S_1 }{d x}(x) , \; x\in (0,L); \qquad S (t,0) = G (t), \; t\in (0,t_0 ), \label{SolopD2}
\end{gather}
 the function $v^0$ being still given by (\ref{Veloc1}) with $A^0$ instead of $h$.

Multiplying equation (\ref{Solop1}) by $u$, equation (\ref{SolopD1}) by $S$, integrating both equations on $(0,t)\times (0,x)$, $(t,x) \in Q_{t_0}$, and adding the obtained identities, we deduce:
\begin{gather}
\frac{1}{2} \int^x_0 \big\{ u^2 + S^2 \big\} (t, \xi ) \ d\xi + \frac{1}{2} \int^t_0  v^0(\tau , x) \big\{ u^2 + S^2 \big\} (\tau , x ) \ d\tau +\notag \\
 \frac{1}{2} \int^t_0  \int^x_0 \p_{x} v^0(\tau , \xi ) \big\{ S^2 -u^2 \big\}  (\tau , \xi ) \ d\xi d\tau +  \int^t_0  \int^x_0 Q^{0, \delta } (\tau ) u (\tau , \xi ) \ d\xi d\tau \notag \\
=\frac{1}{2} \int^x_0 \left\{ S_1^2 + \left( \frac{d S_1 }{dx} \right)^2 \right\} ( \xi ) \ d\xi + \frac{1}{2} \int^t_0  v_{in}(\tau ) \left\{ S_0^2 + G^2 \right\} (\tau ) \ d\tau .
\label{AprioriEst1}
\end{gather}

Now we use (\ref{Data1-}), (\ref{EstDervela}) with (\ref{EstG}) to get
\begin{gather}
\int^t_0  v_{in}(\tau )  G^2  (\tau ) \ d\tau \leq  \int^t_0  \frac{2}{ v_{in}(\tau )} \left\{ \left| \frac{ d S_0 }{d \tau } \right|^2 + | Q^{0, \delta } (\tau ) |^2 \right\} \ d\tau \leq \notag \\
 \int^t_0  \frac{2}{ v_{in}(\tau )} \left| \frac{ d S_0 }{d \tau } \right|^2  \ d\tau +  \frac{4\min \{ t, \delta \} }{v_m} |Q^{0,0} |^2 + \frac{4 V_M^2 R^2}{v_m L^3} t  .\label{LastT}
\end{gather}
The third term on the left hand side of (\ref{AprioriEst1}) is estimated  with the help of (\ref{Itervel3}) and (\ref{Posit1})  as
\begin{gather}
\int^t_0  \int^x_0 \p_{x} v^0(\tau , \xi ) ( u^2 - S^2)   (\tau , \xi ) \ d\xi d\tau \leq  \int^t_0  \int^x_0 \frac{Q^0(\tau)}{A^0 (\tau , \xi)} u^2 (\tau , \xi) \ d\xi d\tau \notag \\
 \leq  4S_M^2 \int^t_0 Q^0(\tau) \int^x_0  \frac{d\xi }{A^0(\tau,\xi)} \ d\tau \le 4S_M^2 \int_0^t \left( v^0(\tau,x)-v_{in}(\tau) \right)\ d\tau  \notag \\
 \le 4S_M^2 V_M t ,\label{PrimoThird}
\end{gather}
while the fourth term  obeys
\begin{gather}
\left| \int^t_0  \int^x_0 Q^{0, \delta } (\tau ) u (\tau , \xi ) \ d\xi d\tau \right| \leq  2S_M L \int_0^t |Q^{0, \delta}(\tau)| \ d\tau   \notag \\
 \le 2S_M L \left( |Q^{0,0}| + \frac{V_M R}{L^{3/2}} \right) t ,\label{ThirdT}
\end{gather}
 thanks to (\ref{EstDervela}) and (\ref{Posit1}).

Now,  setting
$$
y(t) =  \sup_{x\in [0,L]} \left\{ \int^x_0 \big( u^2 + S^2 \big) (t, \xi ) \ d\xi + \int^t_0  v^0(\tau , x) \big\{ u^2 + S^2 \big\} (\tau , x ) \ d\tau \right\} ,
$$
 we may insert (\ref{LastT})-(\ref{ThirdT}) into (\ref{AprioriEst1}) and  use (\ref{Radius}) and (\ref{Small2}) to get the following estimate:
\begin{eqnarray*}
y(t) & \le & \int^L_0 \left\{ S_1^2 + \left( \frac{d S_1 }{d \xi} \right)^2 \right\} ( \xi ) \ d\xi \\
& & +  \int^t_0  v_{in}(\tau ) \left\{ S_0^2 + \frac{2}{ v_{in}^2 (\tau )} \left| \frac{ d S_0 }{d \tau } \right|^2 \right\} (\tau ) \ d\tau  + (\alpha+\beta) t  \\
& \le &  \frac{R^2}{8} + \frac{R^2}{8} \frac{t}{t^*} \le \frac{R^2}{4}  .
\end{eqnarray*}
 Recalling (\ref{Data1-}) and (\ref{Itervel3}), we thus conclude that
\begin{equation}\label{Apriori1}
\mathcal{E}(u) < R .
\end{equation}

Next, we use equation (\ref{Solop1}) and estimate (\ref{Apriori1}) to control $\p_t u$, obtaining
\begin{gather}
\sup_{t\in [0, t_0] }   || \p_t u (t) ||_{L^2 (0,L)} + \sup_{x\in [0, L] } || \p_t u (\cdot , x) ||_{L^2 (0, t_0)}  < \notag \\ R  V_M \big( 2+\frac{1}{L} + \frac{\sqrt{T}}{L^{3/2}} \big) + | Q^{0,0} | (\sqrt{T} + \sqrt{L}).\label{Apriori2}
 \end{gather}
Therefore the image of a nonnegative function from  $\mathcal{S} (t_0 , R)$ remains in $\mathcal{S} (t_0 , R)$ and satisfies the $L^\infty-$ bound (\ref{Posit1}). Therefore our nonlinear map (\ref{Solop}) maps the convex set
\begin{equation}\label{Invariant}
    \mathcal{S}_0 (t_0 , R) = \{  f\in \mathcal{S} (t_0 , R) \, | \, f \ \mbox{satisfies (\ref{Posit1}) a.e.} \}
\end{equation}
into itself.

\vskip10pt
\fbox{STEP~3}
\vskip10pt
 Let $X$ be the intersection of the Banach spaces $ W^{1, \infty} ([0, t_0] ; L^2 (0,L))$, $ L^{\infty} ([0, t_0] ; H^1 (0,L))$,  $W^{1, \infty} (0, L;L^2 (0, t_0))$ and $ L^\infty(0,L;H^1 (0, t_0))$. Clearly, $ \mathcal{S}_0 (t_0 , R) $ is a convex, bounded and closed subset of the Banach space $X$. We apply the Schauder fixed point theorem to prove that the mapping  (\ref{Solop}) admits a fixed point in $ \mathcal{S}_0 (t_0 , R) $. After Step~2, it remains only to prove the sequential continuity of the map (\ref{Solop}). Hence let $\{ A^k \} _{k\in \mathbb{N}}$, be  a  sequence converging in $\mathcal{S}_0 (t_0 , R) $ to $A$. Then we have

\begin{gather} A^k \to A \quad \mbox{ strongly in } \; C([0, t_0 ]; L^2 (0,L))\cap L^2 (0, t_0 ; C([0,L])),
\label{ConvgReg1} \\
Q_k^0 =\frac{v_L  - v_{in} }{\displaystyle \int^L_0 \frac{dx }{A^k (t, x )}}  \to  Q^0 =\frac{v_L  - v_{in} }{\displaystyle \int^L_0 \frac{dx }{A (t, x )}} \quad \mbox{ uniformly  in } \;  C([0 , t_0 ]) , \label{ConvgReg3} \\
 v^k = v_{in} + Q_k^0 \int^x_0  \frac{d\xi }{A^k (t, \xi )}   \to v=v_{in}  + Q^0 \int^x_0  \frac{d\xi }{A (t, \xi )} \notag \\ \mbox{ strongly } \, \mbox{ in } \; C([0, t_0] ; H^1 (0,L)) .
\label{ConvgReg2}
\end{gather}
Let $u^k$ be the solution to (\ref{Solop1})-(\ref{Solop2}), corresponding to $A^k$ and $u$ the solution corresponding to $A$. Then by analogous calculations to those performed in Step~2, we get $u^k \to u$ in $C([0, t_0] ; H^1 (0,L)) \cap L^2 (0, t_0 ; C^1 ([0,L]))$. Using the equation (\ref{Solop1}), we find out that one also has $\p_t u^k \to \p_t u$ in $C([0, t_0] ; L^2 (0,L)) \cap L^2 (0, t_0 ; C ([0,L]))$. Therefore the mapping (\ref{Solop}) is continuous and compact and by Schauder's fixed point theorem there is at least one fixed point $A^\delta \in \mathcal{S}_0 (t_0 , R) $.
 \vskip6pt
 Denoting the corresponding velocity field by $v^\delta$, we have
$$
v^\delta \in W^{1, \infty} (0, t_0 ; H^1 (0,L))\cap L^{\infty} (0, t_0 ; H^2 (0,L))
$$
and since $t\mapsto \p_x v^\delta(t,0)\in W^{1, \infty} (0, t_0)$, we may apply Theorem~\ref{HagenR} in Appendix  to conclude that, in fact,
$$
A^\delta \in C^{1} ([0, t_0] ; H^1 (0,L))\cap C ([0, t_0] ; H^2 (0,L)).
$$
\vskip5pt
It remains to prove uniqueness.

\vskip10pt
\fbox{STEP~4}
\vskip10pt

With the obtained smoothness, uniqueness is easy to establish. It suffices to notice that $L^2 (0, t_0 ; H^1 (0,L))$ perturbation of $v^\delta$ is controlled by the $L^2$-perturbation of $A^\delta$ in $x$ and $t$. Then we use this observation,
regularity of $A^\delta$, and Gronwall's lemma to obtain uniqueness.\vskip5pt

The proof of Theorem~\ref{SmallTimeExist} is now complete.

\section{Global existence of regularized  strong solutions}

Now we suppose that the regularity of the solution or/and the strict positivity of $A^\delta$, stated in (\ref{Posit1}), breaks down    at the time $t_p$. Our goal is to prove that  $t_p=T$.
\begin{theorem}\label{GlobalExistDelta}  Under the hypotheses of Theorem~\ref{SmallTimeExist}, the initial-boundary value problem (\ref{MatShortReg1})-(\ref{MatShortReg4}) has a unique solution
$$    A^\delta \in C^1 ([0, T ] ; H^1 (0, L)) \cap C([0, T ]; H^2 (0,L) ) ,$$
  $$ v^\delta \in C^1 ([0, T] ; H^2 (0, L)) \cap C([0, T ]; H^3 (0,L) ).$$
\end{theorem}

The remaining part of this section is devoted to the proof of Theorem~\ref{GlobalExistDelta}.
\vskip10pt
\fbox{STEP~1}
\vskip10pt

First we recall that,  owing to (\ref{MatShortReg2}) and (\ref{MatShortReg4}),  $v^\delta$ is given by
\begin{equation}
v^\delta (t,x) = v_{in} (t) + \frac{v_L (t) - v_{in} (t)}{\displaystyle \int^L_0 \frac{d\xi }{A^\delta (t, \xi )}}
\int^x_0  \frac{d\xi }{A^\delta (t, \xi )}. \label{VelocDelta1}
\end{equation}
 Next, by (\ref{Data1-})
\begin{equation}
 Q^{\delta} (t) = \frac{v_L (t) - v_{in} (t)}{\displaystyle \int^L_0 \frac{dx }{A^\delta (t, x )}}  >0 \quad \mbox{ in } \; \overline{Q_{t_0}} . \label{ForceDelta1}
\end{equation}
 and we put
\begin{equation}\label{ForceDelta2}
    Q^{cut, \delta} (t) = \left\{
                      \begin{array}{ll}
                        \displaystyle Q^{\delta} (t), \quad \hbox{for} \; \delta \leq t\leq t_0 , &\\ \\
                        \displaystyle Q^{0, 0} + ( Q^{\delta} (\delta ) - Q^{0, 0} ) \frac{t}{\delta}, \quad \hbox{for} \; 0\leq t< \delta .&
                      \end{array}
                    \right.
\end{equation}

\vskip10pt
We point out that the upper bound in (\ref{Posit1}) is independent of $t^*$. Indeed, we prove now that it is valid regardless the length of the time interval.  To this end,  we introduce the solution $S_M^\delta \in W^{2, \infty} (0, t_0 )$  to the Cauchy problem
\begin{equation}
\frac{d S_M^\delta (t)}{d t} = |Q^{0, 0}|  \chi_{ [0,\delta]}(t) , \quad t\in (0, t_0 ) , \quad S_M^\delta(0)=S_M .\label{UpCauchy1}
\end{equation}
Owing to the positivity (\ref{ForceDelta1}) of $Q^\delta$, $S_M^\delta$ is a supersolution to (\ref{MatShortReg1})-(\ref{MatShortReg1a}) with $S_M^\delta(0)=S_M\ge S_1(x)=A^\delta(0,x)$ for $x\in [0,L]$ and $S_M^\delta(t)\ge S_M^\delta(0) \ge S_0(t) =A^\delta(t,0)$ for $t\in (0,t_0)$. The comparison principle then implies
\begin{equation}\label{LowDeltabd2}
    A^\delta (t,x) \leq S_M^\delta (t) \leq S_M + \delta | Q^{0,0} |.
\end{equation}
 This proves the upper bound and, in addition, the estimate is independent of the length of the time interval.
\vskip5pt
Therefore, by Jensen's inequality,  we may infer that
\begin{gather}
0< Q^{\delta} (t) \leq \frac{ V_M }{\displaystyle \int^L_0 \frac{d\xi }{A^\delta (t, \xi )}} \leq \frac{V_M}{L^2}
\int^L_0 A^\delta (t, \xi ) \ d\xi \notag \\
\leq \frac{V_M}{L} ( S_M + \delta | Q^{0,0} |) , \label{EstDervelDelta} \\
| G^\delta (t) | = | \p_x A^\delta  (t, 0) |  = \left| -\frac{1}{v_{in} (t)} \left( \frac{d S_0 (t)}{d t} + Q^{cut ,\delta } (t)\right) \right| \notag \\
\le \frac{1}{v_m} \left(  \left| \frac{d S_0 (t)}{d t} \right| + |Q^{0,0} | + \frac{V_M}{L} ( S_M + \delta | Q^{0,0} |) \right),\label{EstGDelta}
\end{gather}
for any $t < t_p$, where $t_p$ is the critical time at which $A^\delta (t,x)$ attains zero for some $x$.

\vskip10pt
\fbox{STEP~2}
\vskip10pt

Having established that $\p_x A^\delta  (t, 0)$ is bounded in $L^\infty (0,t)$ independently of $t$, we now look for a strictly positive lower bound for $ A^\delta $.

\begin{proposition}\label{logbound}
There are constants $C_1$ and $C_2$, independent of the length of the time interval and of $\delta$, such that we have
\begin{equation}\label{DerlogBd}
    C_1 \leq \p_x  \log A^\delta  (t, x) \leq C_2 \qquad \mbox{ on } \quad \overline{Q_{t_0}} .
\end{equation}
\end{proposition}

\begin{proof}
We notice that $y = \p_x \log A^\delta$ satisfies the equation
\begin{equation}\label{LogEq}
\p_t y + v^\delta  \p_x y + \frac{y}{A^\delta}  ( Q^{\delta} - Q^{cut , \delta} )=0  \; \mbox{ in }\; Q_{t_p},
\end{equation}
and
\begin{equation}\label{condDelta}
    y(0,x) = \frac{d \log S_1}{d x}(x) , \quad x\in (0,L) \, , \qquad y(t,0) = \frac{G^\delta (t)}{S_0 (t)} , \quad  t\in (0,t_p).
\end{equation}
 The function $F(t,x) = ( Q^{\delta} - Q^{cut, \delta})(t)/A^\delta (t,x)$ vanishes for $t\ge \delta$ and, due to (\ref{Posit1}), satisfies the following bound
\begin{eqnarray}
\int_0^t || F(\tau) ||_{L^\infty (0,L)} \ d\tau & = & \int_0^{\min{\{t,\delta\}}} || F(\tau) ||_{L^\infty (0,L)} \ d\tau \nonumber \\
& \leq & \frac{8\delta }{S_m} \left( | Q^{0,0} | + || Q^\delta ||_{L^\infty (0, \delta)} \right) \label{BdF}
\end{eqnarray}
 for every $t$ since $\delta \leq t_0$.

 We next introduce the solution $y_m$ to the ordinary differential equation
\begin{gather}
\frac{d y_m }{ d t}(t) = || F (t) ||_{L^\infty (0,L)}\ y_m(t) \quad \mbox{ in } \; (0,T), \label{Barrier1}
\end{gather}
with initial condition
\begin{equation}\label{pim}
y_m(0) = \min{ \left\{ \min_{[0,L]}{\left\{ \frac{d\log{S_1}}{dx} \right\}} , - G_M\right\} } \le 0 ,
\end{equation}
with
\begin{equation}\label{poum}
G_M = \frac{1}{v_m} \left[ \left\| \frac{dS_0}{dt} \right\|_{L^\infty(0,T)} + |Q^{0,0}| + \frac{V_M}{L} ( S_M + \delta | Q^{0,0} |) \right] .
\end{equation}
Owing to (\ref{EstGDelta}), (\ref{condDelta}), (\ref{Barrier1}), and (\ref{pim}), $y_m$ is nonpositive and thus a subsolution to (\ref{LogEq}), and satisfies $y_m(t)\le y_m(0) \le y(t,0)$ for $t\in (0,t_p)$ and $y_m(0)\le y(0,x)$ for $x\in (0,L)$. The comparison principle then entails that
$$
y(t,x) \ge y_m (t) = y_m (0) \exp{ \int^t_0 || F(\tau ) ||_{L^\infty (0,L) } \ d\tau } .
$$
Since $y_m(0)\le 0$, we deduce from (\ref{BdF}) that $y(t,x) \ge C_1$ for some positive constant $C_1$, independent of $\delta$ and $t_0$.

 Similarly, let $Y_M$ be the solution to the ordinary differential equation

\begin{gather}
\frac{d Y_M }{ d t}(t) = || F (t) ||_{L^\infty (0,L)} \ Y_M(t) \quad \mbox{ in } \; (0,T), \label{Barrier1c}
\end{gather}
with initial condition
\begin{equation}\label{pam}
Y_M(0) = \max{ \left\{ \max_{[0,L]}{\left\{ \frac{d\log{S_1}}{dx} \right\}} , G_M\right\} } \ge 0 ,
\end{equation}
the constant $G_M$ being defined in (\ref{poum}). It follows from (\ref{EstGDelta}), (\ref{condDelta}), (\ref{Barrier1c}), and (\ref{pam}) that $Y_M$ is a supersolution to (\ref{LogEq}) which satisfies $Y_M(t)\ge Y_M(0) \ge y(t,0)$ for $t\in (0,t_p)$ and $Y_M(0)\ge y(0,x)$ for $x\in (0,L)$. Using once more the comparison principle and (\ref{BdF}), we conclude that
$$
y(t,x) \le Y_M (t) = Y_M(0) \exp{ \int^t_0 || F(\tau ) ||_{L^\infty (0,L) } \ d\tau } \le C_2 ,
$$
the constant $C_2$ being independent on $\delta$ and $t$.
\end{proof}

\medskip

Since  $\log A^\delta  (t, 0) = \log S_0 (t)  \ge \log{S_m} > -\infty$ by (\ref{Data1-}) and
$$ \log A^\delta  (t, x) = \log S_0 (t) + \int^x_0 \p_\xi \log A^\delta  (t, \xi ) \ d\xi, $$
we infer from (\ref{DerlogBd}) that
\begin{equation}\label{LinftyLogbd}
   ||  \log A^\delta  ||_{L^\infty (0, t_0 ; W^{1, \infty } (0,L)) } \leq C .
\end{equation}
Therefore $A^\delta$ is strictly positive on $\overline{Q_{t}}$ for all  $t\leq t_p$, and bounded from below and from above by a constant which is independent of both $t$ and $\delta$.

 \vskip10pt
\fbox{STEP~3}
\vskip10pt

Having established (\ref{LinftyLogbd}) we easily obtain the following estimates:
\begin{gather}
|| \p_t A^\delta ||_{L^\infty (Q_{t})} + || \p_x A^\delta ||_{L^\infty (Q_{t})} \leq C,
\label{DerivDeltaA} \\
    ||  v^\delta ||_{L^\infty (Q_{t})} + || \p_x v^\delta ||_{L^\infty (Q_{t})} + || \p_{x}^2 v^\delta ||_{L^\infty (Q_{t})} + || \p^2_{xt} v^\delta ||_{L^\infty (Q_{t})} \leq C,\label{DerivDeltav}
\end{gather}
where $C$ is again independent of $t$ and $\delta$.\vskip5pt
The estimate (\ref{DerivDeltav}) guarantees that the coefficients in equations (\ref{MatShortReg1})-(\ref{MatShortReg1a}) remain regular, whence Theorem~\ref{HagenR} is applicable. Consequently,  $A^\delta$ remains bounded in $C^1 ([0, t]; H^1 (0,L))\cap C ([0, t]; H^2 (0,L))$.\vskip5pt
We conclude that,  for all $t$, $A^\delta$ is bounded from below by a positive constant, independent of $t$, and the norm of $A^\delta$ in $C^1 ([0, t]; H^1 (0,L))\cap C ([0, t]; H^2 (0,L))$ remains bounded by a constant, also independent of $t$, that may, however, depend on $\delta$. The maximal solution therefore  extends to $[0,T]$, and, in fact, we have established the existence of a unique strictly positive solution $A^\delta$ on $(0,T) \times (0,L)$. The corresponding velocity $v^\delta$ is given by (\ref{Veloc1}), with $h=A^\delta$.

\vskip4pt This completes the proof of Theorem~\ref{GlobalExistDelta} .

\section{Existence of a unique strong solution}

At this stage, we are ready to establish the main result of the paper.

\begin{theorem}\label{GlobalExist}  Under the hypotheses (\ref{Data1-}), (\ref{Data1}), the initial-boundary value problem (\ref{Mat1})-(\ref{Mat4}) possesses a unique (strong) solution $A$, $v$ on ${\bar Q}_T$, belonging to the class
$$A, \ \partial_t A, \ \p_x A \in L^\infty((0,L) \times (0,T)) ,$$
$$ v, \ \partial_t v , \ \partial^2_{t,x} v , \ \partial^2_{x,x} v \in L^\infty((0,L) \times (0,T)). $$
\end{theorem}

\begin{proof}
We just recall the estimates obtained in the proof of Theorem~\ref{GlobalExistDelta}, valid independently of $\delta$:
\begin{gather}
0< C_1 \leq A^\delta (t,x) \leq C\qquad \mbox{ in  } \quad \overline{Q_T},
\label{PointA} \\
    || \p_x A^\delta ||_{L^\infty (Q_{T})} + || \p_t A^\delta ||_{L^\infty (Q_{T})} \leq C.
\label{DerivA}
\end{gather}
Since $L^\infty (Q_T) $ is the dual space of the separable Banach space $L^1 (Q_T)$, Alaoglu's weak$^*$ compactness theorem gives weak$^*$ sequential compactness.  Therefore there exist $A$ and $v$ such that
\begin{gather}
A^\delta \to A \qquad \mbox{uniformly in } \quad \overline{Q_T} , \; \mbox{ as } \; \delta \to 0,
\label{ConvUnif} \\
\p_x A^\delta \rightharpoonup \p_x A \qquad \mbox{weakly}^* \, \mbox{ in } \quad L^\infty ({Q}_T ), \; \mbox{ as } \; \delta \to 0,
\label{WeakDerx} \\
   \p_t A^\delta \rightharpoonup \p_t A \qquad \mbox{weakly}^* \, \mbox{ in } \quad L^\infty ({Q}_T ), \; \mbox{ as } \; \delta \to 0,
\label{WeakDert} \\
v^\delta \to v = v_{in} (t) + \frac{v_L (t) - v_{in} (t)}{\displaystyle \int^L_0 \frac{d\xi }{A (t, \xi )}}
\int^x_0  \frac{d\xi }{A (t, \xi )}\, \mbox{uniformly in } \, \overline{Q_T} , \, \mbox{ as } \ \delta \to 0.
\label{ConvUnifVel} \\
\p_x v^\delta \to \p_x v =  \frac{v_L (t) - v_{in} (t)}{\int^L_0 \displaystyle \frac{d\xi }{A (t, \xi )}}
 \frac{1}{A (t, x )}\, \mbox{uniformly in } \, \overline{Q_T} , \, \mbox{ as } \ \delta \to 0,
\label{ConvUnifVelDer}
\end{gather}
at least for suitable subsequences. Obviously, $(A, v)$ solves the system (\ref{Mat1})-(\ref{Mat4}). Moreover, by virtue of the interior regularity, the equations (\ref{Mat1})-(\ref{Mat2}) are satisfied pointwise.  Finally, according to  the smoothness of $A$ and $v$, the proof of uniqueness is straightforward.
\end{proof}

\medskip

\begin{remark}
The standard way of proving uniqueness relies on Gronwall's inequality. For any bounded time interval, small $L^2$-perturbations of the data $v_L$, $v_{in}$, $S_0$ in the $L^2-norm$ result in the corresponding variation of the solution in the same norm, that may depend exponentially on the length of the time interval. Better estimates would require refined analytical arguments.
\end{remark}

\section{Appendix}

Here we recall the result from \cite{HagenRenardy99}, which is used in this paper:

\begin{theorem}\label{HagenR}
Let $f$ and $p$ be given continuous functions defined on $[0, t_0 ]\times [0,b]$, $b>0$. Let $u^0 \in C([0,b])$ and $u^b \in C([0, t_0 ])$ and let us suppose\footnote{Hagen and Renardy suppose $\p_x p (\cdot, b), \p_x f (\cdot , b) \in   H^{1} (0, t_0 )$ as well. This does not seem to be necessary.} that
\begin{gather}
 p, f \in W^{1, \infty} ([0, t_0] ; H^1 (0,b))\cap L^{\infty} ([0, t_0] ; H^2 (0,b)), \notag \\
 \mbox{ are such that} \quad  \p_x p (\cdot , 0) \; \mbox{and} \;  \p_x f (\cdot , 0)\in   H^{1} (0, t_0 );  \label{Hyp1} \\
    p< 0 \quad \mbox{ on } \quad [0, t_0 ]\times [0,b] ;  \label{Hyp2} \\
u^0 \in H^2 (0,b), \quad u^b \in H^2 (0, t_0 );  \label{Hyp3} \\
u^0 (0) = u^b (0); \quad \p_t u^b (0) = p(0,0) \p_x u^0 (0) + f(0,0).  \label{Hyp4}
\end{gather}
Then the boundary-initial value problem
\begin{gather}
\p_t u = p(t,x) \p_x u + f(t,x) , \quad (x,t) \in  (0, t_0 )\times (0,b), \label{HypEq1} \\
u(0,x) = u^0 (x), \; x\in (0,b); \qquad u(t,0) = u^b (t), \; t\in (0,t_0 ). \label{HypEq2}
\end{gather}
has a solution
\begin{equation}
\label{SolHyp1}
    u\in C^1 ([0, t_0 ] ; H^1 (0, b)) \cap C([0, t_0 ]; H^2 (0,b) ),
\end{equation}
which is unique in $W^{1, \infty} (0, t_0 ; L^2 (0,b))\cap L^{\infty} (0, t_0 ; H^1 (0,b))$.
\end{theorem}



\end{document}